\documentclass[a4paper,12pt]{article}
\title{Torsion points and height jumping in higher-dimensional families of abelian varieties}

\usepackage{amsmath, amssymb, mathrsfs, amsthm, geometry, mathtools, graphicx}
\usepackage{hyperref}
\usepackage[all]{xy}

\usepackage{enumitem}

\usepackage{cleveref}
\let\oref\ref
\AtBeginDocument{\renewcommand{\ref}[1]{\cref{#1}}}

\newcommand{\on}[1]{\operatorname{#1}}
\newcommand{\bb}[1]{{\mathbb{#1}}}
\newcommand{\cl}[1]{{\mathscr{#1}}}
\newcommand{\ca}[1]{{\mathcal{#1}}}
\newcommand{\bd}[1]{{\mathbf{#1}}}

\newcommand{\Span}[1]{\left<#1\right>}

\newcommand{\abs}[1]{\lvert#1\rvert}
\newcommand{\aabs}[1]{\lvert\lvert#1\rvert\rvert}

\newcommand{\ra}{\rightarrow}

\newcommand{\hra}{\hookrightarrow}
\newcommand{\sub}{\subseteq}

\theoremstyle{definition}
\newtheorem{definition}{Definition}[section]
\newtheorem{conjecture}[definition]{Conjecture}

\theoremstyle{plain}
\newtheorem{proposition}[definition]{Proposition}
\newtheorem{lemma}[definition]{Lemma}
\newtheorem{theorem}[definition]{Theorem}
\newtheorem{corollary}[definition]{Corollary}

\theoremstyle{remark}

\newtheorem{remark}[definition]{Remark}

\renewcommand{\phi}{\varphi}

\author{David Holmes}
\date{\today}

\newcounter{nootje}
\renewcommand\check[1]{[*\thenootje]\marginpar{\tiny\begin{minipage}
{20mm}\begin{flushleft}\thenootje : 
#1\end{flushleft}\end{minipage}}\addtocounter{nootje}{1}}

\newcommand{\beq}{\begin{equation}}
\newcommand{\eeq}{\end{equation}}
\newcommand{\beqs}{\begin{equation*}}
\newcommand{\eeqs}{\end{equation*}}

\begin{document}
\maketitle
\begin{abstract} 
In 1983 Silverman and Tate showed that the set of points in a 1-dimensional family of abelian varieties where a section of infinite order has `small height' is finite. We conjecture a generalisation to higher-dimensional families, where we replace `finite' by `not Zariski dense'. We show that this conjecture would imply the Uniform Boundedness Conjecture for torsion points on abelian varieties. We then prove a few special cases of this new conjecture. 
\end{abstract}
\newcommand{\tor}{\bd{T}}
\tableofcontents

\section{Introduction}

\newcommand{\qq}{K}
\newcommand{\cseq}{\mathfrak{K}}

The uniform boundedness conjecture predicts that the number of rational torsion points on an abelian variety over $\bb{Q}$ is bounded uniformly in the dimension of the abelian variety. It is a theorem of Mazur for elliptic curves (with generalisations to number fields by Kamienny, Merel and others), but is wide open for higher dimensional abelian varieties. 

An easy consequence of this conjecture is that, given a family of abelian varieties and a section of infinite order, the set of rational points in the base where the section becomes torsion is not Zariski dense. This is known unconditionally for families of elliptic curves by Mazur's theorem, but not in general. However, it is also known unconditionally for abelian varieties of any dimension if the base of the family has dimension 1 - this is a consequence of a theorem of Silverman \cite{Silverman1983Heights-and-the} and Tate \cite{Tate1983Variation-of-th} (with refinements by a number of authors - \cite{Lang1983Fundamentals-of}, \cite{Call1986Local-heights-o}, \cite{Green1989Heights-in-fami}, \cite{David-Holmes2014Neron-models-an}
).

The first main result of this paper is that the uniform boundedness conjecture is in fact \emph{equivalent} to showing that the set of points where a section of infinite order becomes torsion is not Zariski dense (\ref{thm:main_equivalence}). In particular, if we could generalise the theorem of Silverman and Tate to families of arbitrary dimension (cf. \ref{conjecture_ST}), this would imply uniform boundedness. The proof of the equivalence is not difficult, it mainly uses induction on dimension and the fact that the stack of principally polarised abelian varieties of fixed dimension is a noetherian Deligne-Mumford stack. 

By a recent result of Cadoret and Tamagawa, the uniform boundedness conjecture for abelian varieties is equivalent to the same conjecture for jacobians of curves. Using this, we can also show that the restriction of \ref{conjecture_ST} to jacobians implies the full uniform boundedness conjecture. 

In fact Silverman and Tate (\cite{Silverman1983Heights-and-the}, \cite{Tate1983Variation-of-th})show not only the finiteness of the number of points where the section becomes torsion, but (stronger) the finiteness of the set of points where the section has bounded N\'eron-Tate height. Motivated by this, we wonder whether the same may be true for families of higher dimension. Writing $\hat{\on{h}}$ for the N\'eron-Tate height, we define
\begin{definition}
Given a variety $S/K$, an abelian scheme $A/S$, and a section $\sigma \in A(S)$, we say $(A/S, \sigma)$ has \emph{sparse small points} if for all $d \in \bb{Z}_{\ge 1}$ there exists $\epsilon >0$ such that the set
\begin{equation*}
{\tor}_\epsilon(d) = \{s\in S(\bar{\qq}) | [\kappa(s): \qq] \le d \text{ and }\hat{\on{h}}(\sigma(s)) \le \epsilon\}
\end{equation*}
is not Zariski dense in $S$.  
\end{definition}
We tentatively propose
\begin{conjecture}\label{conj:sparse_small_points}\label{conjecture_SSP}
Every pair $(A/S, \sigma)$ with $\sigma$ of infinite order has sparse small points. 
\end{conjecture}
If $S$ has dimension 1 Silverman and Tate actually show this to be true for every $\epsilon$, not just sufficiently small $\epsilon$, but we are hesitant to conjecture this in general. Since torsion points have height zero the above conjecture would imply that sets of points where a section of infinite order becomes torsion are not Zariski dense, and hence (by \ref{thm:main_equivalence} mentioned above) the uniform boundedness conjecture. The conjecture is true for families of elliptic curves, subject to a conjecture of Lang on lower bounds for heights in families (\ref{lem:ST_for_EC}). 

In the second half of this paper we prove some special cases of \ref{conj:sparse_small_points}. In \cite{Holmes2016Quasi-compactne} we proved that torsion points are sparse for families of jacobians which admit N\'eron models. Here we consider again families admitting N\'eron models, but the proof is far more involved, and we are forced to impose some additional assumptions in order for our methods to work, see \ref{thm:main_sparsity}. Most important is a geometric condition on the  base $S$; a precise statement can be found in \ref{thm:main_sparsity}, but here we note that it is satisfied whenever $\on{dim}_\bb{Q} \on{Pic} (S) \otimes_\bb{Z} \bb{Q}  = 1$, yielding a slight simplification of \ref{thm:main_sparsity}:
\begin{theorem}
Let $S/\bb{Q}$ be a projective variety with $\on{Pic}(S) \times_\bb{Z} \bb{Q} \cong \bb{Q}$, and let $U \sub S$ be a dense open subscheme. Let $C/S$ be a family of nodal curves, smooth over $U$, with $C$ regular. Write $J$ for the jacobian of $C_U/U$, and let $\sigma \in J(S)$ be a section of infinite order corresponding to the restriction to $C_U$ of a divisor on $C$ supported on some sections of $C/S$. Assume that $S$ has a proper regular model over $\bb{Z}$ over which $C$ has a proper nodal model. Assume that $J$ admits a N\'eron model over $S$. Then $(J/U, \sigma)$ has sparse small points. 
%
\end{theorem}
For the sake of those readers not familiar with the theory of N\'eron models over higher dimensional bases developed in \cite{Holmes2014Neron-models-an} we note that the assumption that $J$ admit a N\'eron model over $S$ can be replaced by the assumption that $S\setminus U$ (with reduced induced scheme structure) is smooth over $\bb{Q}$, or that the fibres of $C \ra S$ have tree-like dual graphs.

We conclude the introduction by giving an outline of the proof of \ref{thm:main_sparsity}. In \cite{David-Holmes2014Neron-models-an} we gave a new proof of the theorem of Silverman and Tate over curves, and in this article we follow essentially the same strategy of proof in higher dimensions. Namely, given a family of jacobians $J/S$ and a section $\sigma \in J(S)$ of infinite order, we consider the function 
\begin{equation*}
\hat{\on{h}}_\sigma\colon S(\bar{\bb{Q}}) \ra \bb{R}_{\ge 0}; \;\; s \mapsto \hat{\on{h}}(\sigma(s)). 
\end{equation*}
We decompose $\hat{\on{h}}_\sigma$ as a sum of two functions, 
\begin{equation*}
\hat{\on{h}}_\sigma = \on{h}_\cl{L} + j
\end{equation*}
where $\on{h}_\cl{L}$ is a Weil height on $S$ with respect to a certain line bundle, and $j$ is an `error term', called the \emph{height jump}. When $S$ is a curve, we were able to show that the bundle $\cl{L}$ is ample and the function $j$ is bounded, from which the theorem of Silverman and Tate easily follows. 

Unfortunately, in extending from curves to arbitrary varieties we encounter two substantial technical difficulties; the line bundle $\cl{L}$ is not in general ample, and the function $j$ does not seem to be bounded\footnote{More precisely, we can write $j= \sum_p j_p$ as a sum of local terms indexed by primes of $\bb{Q}$, and each $j_p$ is unbounded on $\bb{Q}_p$-valued points. We may hope that for $\bb{Q}$-valued points all these local jumps are bounded - indeed for torsion points on elliptic curves this should follow from Mazur's theorem - but we cannot prove this at present. }. The assumptions in \ref{thm:main_sparsity} allow us to get around these problems. We will show that the jump is bounded if the jacobian $J$ has a N\'eron model over a suitable compactification of $S$, and that $\cl{L}$ is ample if $S$ has sufficiently simple geometry (cf. our condition that $\on{dim}_\bb{Q} \on{Pic}(\bar{S})\otimes_\bb{Z} \bb{Q} = 1$). 

Many thanks to Owen Biesel, Maarten Derickx, Wojciech Gajda, Ariyan Javanpeykar, Robin de Jong and Pierre Parent for helpful comments and discussions. Particular thanks to Owen and Robin for allowing me to include the material in \ref{sec:aligned_jump_vanishes} after we removed it from our joint paper \cite{David-Holmes2014Neron-models-an} to avoid references to then-unpublished results. 

%
%

\subsection{Conventions}
Given a field $k$, by a \emph{variety over $k$} we mean an integral separated $k$-scheme of finite type. We fix throughout a field $\qq$, which is either the field of rational numbers or the field $\bb{F}_q(T)$ for some prime power $q$. We fix an algebraic closure $\bar{\qq}$ of $\qq$. We write $\kappa(p)$ for the residue field of a point $p$.

\section{Reformulating the uniform boundedness conjecture}

We begin by recalling the uniform boundedness conjecture:
\begin{conjecture}[The strong uniform boundedness conjecture for abelian varieties]\label{conj:SUB}
Fix an integer $g \ge 0$. There exists a constant $B = B(g) \in \bb{Z}_{>0}$ such that for every $g$-dimensional principally polarised abelian variety $A/\qq$ and point $p \in A(\qq)$, we have that either $p$ is of infinite order, or that the order of $p$ is less than $B$. 
\end{conjecture}
Note that this is equivalent to the usual formulation of the strong uniform boundedness conjecture (see for example \cite[conjecture 2.3.2]{Silverberg2001Open-questions-}); Zarhin's trick reduces the general case to that of principally polarised abelian varieties, by Weil restriction we can obtain a bound for abelian varieties over finite extensions of $\qq$ which is uniform in the dimension of the variety and the degree of the field, and finally a bound on the order of torsion points and on the dimension implies a bound on the size of the torsion subgroup. 

Our first main result (\ref{thm:main_equivalence}) is that \ref{conj:SUB} is equivalent to the following conjecture:
\begin{conjecture}\label{conjecture_ST}
Let $S/\qq$ be a variety and let $A/S$ be an abelian scheme. Let $ d \ge 1$ be an integer. Let $\sigma \in A(S)$ be a section of infinite order. Define
\begin{equation*}
{\tor}(d) = \{s\in S(\bar{\qq}) | [\kappa(s):\qq] \le d \text{ and }\sigma(s) \text{ is torsion in }A_s(\bar{\qq})\}. 
\end{equation*}
Then $\tor(d)$ is not Zariski dense in $S$.  
\end{conjecture}
In the case when the base scheme $S$ has dimension 1, this conjecture follows from a theorem of Silverman \cite{Silverman1983Heights-and-the}, see also \cite{Tate1983Variation-of-th}, \cite{Lang1983Fundamentals-of}, \cite{Call1986Local-heights-o}, \cite{Green1989Heights-in-fami}, \cite{David-Holmes2014Neron-models-an} for various strengthened versions. In \ref{sec:small_points} we will discuss a variant which looks not only at torsion points but at all points of `small height'. 


\begin{theorem}\label{thm:main_equivalence}
\Cref{conj:SUB} is equivalent to \ref{conjecture_ST}. 
\end{theorem}

It is easy to see that \ref{conj:SUB} implies \ref{conjecture_ST}; we must show the converse. Before giving the proof we need a lemma and a definition. 

\begin{lemma}\label{lem:fundamental_induction}
Assume \ref{conjecture_ST}. Let $S/\qq$ be a variety, and let $A/S$ be an abelian scheme with a section $\sigma \in A(S)$. Fix $d \in \bb{Z}_{>0}$. Then there exists an integer $c=  c(d)>0$ such that for every $L/\qq$ of degree at most $d$ and every point  $s \in S(L)$, either $\sigma(s)$ has infinite order, or $\sigma(s)$ is torsion of order less than $c$. 
\end{lemma}
\begin{proof}
We proceed by induction on the dimension of $S$. If $\dim S = 0$ then $S$ has only finitely many $\bar{\qq}$-points, so the result is immediate. 

In general, we fix an integer $\delta >0$ and assume the lemma holds for every variety $S$ of dimension less than $\delta$. Now let $S$ have dimension $\delta$. If $\sigma$ is torsion (say of order $c_0$) then for every $s \in S(\overline{\qq})$ it holds that $c_0\sigma(s) = 0$, and we are done. As such, we may and do assume that $\sigma$ has infinite order. We apply \ref{conjecture_ST} to obtain a proper closed subscheme $Z \sub S$ such that 
$$ \{p\in S(L) | [L:\qq]\le d \text{ and }\sigma(p) \text{ is torsion in }A_p(L)\} \sub Z(\overline{\qq}). $$
Now $Z$ has only finitely many irreducible components, and each has dimension less than $\delta$. We are done by the induction hypothesis. 
\end{proof}

\begin{definition}
Fix an integer $g>0$. Let $\cl{A}_{g}$ denote the moduli stack of PPAV of dimension $g$. Then $\cl{A}_g$ is a separated Deligne-Mumford stack of finite type over $\bb{Z}$. Given also an integer $n\ge 0$, let $\cl{A}_{g,n}$ denote the moduli stack of PPAV of dimension $g$ together with a collection of $n$ ordered marked sections (not assumed distinct). 

We have natural maps $\phi_n\colon \cl{A}_{g,n+1} \ra \cl{A}_{g,n}$ given by forgetting the last section. The map $\phi_n$ has $n$ natural sections $\tau_i$, given by `doubling up' the sections $\sigma_i$. Then $(\phi_n\colon \cl{A}_{g,n+1} \ra \cl{A}_{g,n}, \tau_1, \cdots, \tau_n)$ is the `universal PPAV with $n$ marked sections'. In particular, each $\cl{A}_{g,n}$ is a Deligne-Mumford stack, separated and of finite type over $\bb{Z}$. 
\end{definition}

\begin{proof}[Proof of \ref{thm:main_equivalence}]
Consider the universal map $\phi_{2}\colon \cl{A}_{g,2} \ra \cl{A}_{g,1}$ with its tautological section $\sigma_1$. The basic idea is to apply \ref{lem:fundamental_induction} to this family, but we must be careful since $\cl{A}_{g,1}$ is a stack and not a scheme. However, since $\cl{A}_{g,1}$ is noetherian we can apply \cite[theorem 16.6]{Laumon2000Champs-algebriq} (every noetherian Deligne-Mumford stack admits a finite, surjective and generically \'etale morphism from a scheme) to construct a (separated) scheme $S$ of finite type over $K$, a map $S \ra \cl{A}_{g,1}$ and an integer $d \ge 1$ such that for every field-valued point $s\colon \on{Spec} L \ra \cl{A}_{g,1}$ there is an extension $M/L$ of degree at most $d$ and an $M$-valued point of $S$ lying over $s$. We are then done by \ref{lem:fundamental_induction} applied to the pullback of $\cl{A}_{g,2}$ and $\sigma_1$ to (the underlying reduced subscheme of each irreducible component of) $S$. 
\end{proof}

\subsection{Extensions and generalisations}

\begin{remark}
Using a recent result of Cadoret and Tamagawa \cite{Cadoret2013Note-on-torsion}, we can reduce further to the case of families of curves:
\begin{conjecture}\label{conjecture_curves}
Let $S/\qq$ be a variety and let $C/S$ be a proper smooth curve with jacobian $J/S$. Let $ d \ge 1$ be an integer. Let $\sigma \in J(S)$ be a section of infinite order. Define
\begin{equation*}
{\tor}(d) = \{s\in S(\bar{\qq}) | [\kappa(s):\qq] \le d \text{ and }\sigma(s) \text{ is torsion in }J_s(\bar{\qq})\}. 
\end{equation*}
Then $\tor(d)$ is not Zariski dense in $S$.  
\end{conjecture}
The equivalence of \ref{conj:SUB} with \ref{conjecture_curves} may be proven in an almost identical fashion to the equivalence of \ref{conj:SUB} with \ref{conjecture_ST}, after first appealing to the main result of \cite{Cadoret2013Note-on-torsion} to reduce \ref{conj:SUB} to the case of curves. We omit the details. 
\end{remark}

\begin{remark}
It is possible to simultaneously `specialise' both conjectures to again obtain equivalent statements; for example:
\begin{itemize}
\item
\ref{conj:SUB} for elliptic curves is equivalent to \ref{conjecture_curves} for elliptic curves over base schemes of dimension at most $2 = \on{dim}\overline{\ca{M}}_{1,2}$;
\item 
\ref{conj:SUB} for abelian surfaces is equivalent to \ref{conjecture_curves} for genus 2 curves over base schemes of dimension at most $5 = \on{dim}\overline{\ca{M}}_{2,2}$;
\item 
\ref{conj:SUB} for abelian varieties of dimension $g$ is equivalent to \ref{conjecture_curves} for curves of genus $G \coloneqq 1+6^g(g-1)!\frac{g(g-1)}{2}$ over base schemes of dimension at most $\on{dim}\overline{\ca{M}}_{G, 2\lceil\frac{G}{2}\rceil}$ (see \cite[theorem 1.2]{Cadoret2013Note-on-torsion} for the origin of this expression for $G$). 
\end{itemize}
The proofs are identical and the possible variations numerous, so we do not give further details. One can also reduce further to the case of families of curves admitting compactifications with at-worst nodal singularities, etc. 

\end{remark}

\begin{remark}
Let $A/S$ be an abelian scheme, $\sigma$ a section of infinite order, and $f\colon S' \ra S$ an alteration (a proper surjective generically finite map of varieties). Then \ref{conjecture_ST} holds for $\sigma$ in $A/S$ if and only if it holds for $f^*\sigma$ in $A\times_S S'/S'$. A corresponding statement holds for \ref{conjecture_curves}. This has several convenient corollaries; for example, it is enough to prove \ref{conjecture_curves} for families of curves $C/S$ admitting stable models over some compactification of $S$. In \ref{sec:special_cases} we will study \ref{conjecture_curves} in detail in the case of such families. 
\end{remark}

\section{Conjecture on sparsity of small points}\label{sec:special_cases}
\label{sec:small_points}

In the introduction we proposed (\ref{conj:sparse_small_points}) that the sets of points where a section of infinite order has small height is not Zariski dense. It is clear that this conjecture implies \ref{conjecture_ST}. The remainder of this paper will be devoted to proving special cases of \ref{conj:sparse_small_points}, first for elliptic curves and then for certain special families of abelian varieties of higher dimension (\ref{thm:main_sparsity}). 

\subsection{\Cref{conjecture_SSP} for elliptic curves}
To suggest that \ref{conjecture_SSP} is not completely unreasonable, we show that a conjecture of Lang implies \ref{conjecture_SSP} for elliptic curves and $d=1$. Moreover, this conjecture of Lang is in fact a theorem (see \cite{Hindry1988The-canonical-h}) over global function fields, yielding an unconditional proof of \ref{conjecture_SSP} for elliptic curves and $d=1$ over global function fields. 

We begin by recalling Lang's conjecture:
\begin{conjecture}(Lang, \cite{Lang1978Elliptic-curves} or \cite[conjecture F.3.4 (a)]{Hindry2000Diophantine-geo}). \label{conj:lang}
Fix a global field $k/K$. There exists a constant $c = c(k)>0$ such that for all elliptic curves $E/k$ and all non-torsion points $a \in E(k)$, we have
\begin{equation*}
\hat{\on{h}}(a) \ge c\cdot \log N_{k/\qq}\Delta_{E/k}.
\end{equation*}
Here $\Delta_{E/k}$ is the discriminant, and $N_{k/\qq}$ denotes the norm down $K$. 
\end{conjecture}
\begin{lemma}\label{lem:ST_for_EC}
If $\on{char}K = 0$ then assume \ref{conj:lang} holds. \Cref{conjecture_SSP} is true assuming that $A/S$ is a family of elliptic curves and that $d=1$. 
\end{lemma}
\begin{proof}

Let $\Sigma$ denote the finite set of elliptic curves over $K$ with everywhere good reduction, and let $b>0$ denote the smallest height of a non-torsion $K$-point appearing on any curve in $\Sigma$, or set $b=1$ if no such exists. Let $c$ be the constant from Lang's conjecture, and let $m$ denote the infimum of the values taken by the expression $c\cdot \log \Delta_{E/K}$ as $E$ runs over all elliptic curves over $K$ with at least one place of bad reduction; this infimum is achieved (since there are only finitely many curves of bounded discriminant) and is positive (by our bad-reduction assumption). 


 Then setting $\delta = \min (b, m)$ we find for all $\delta > \epsilon \ge 0$ we have 
\begin{equation*}
\tor_\epsilon(1) = \tor(1). 
\end{equation*}

Now \ref{conj:SUB} is known for elliptic curves over the rationals by work of Mazur \cite{Mazur1977Modular-curves-}, \cite{Mazur1978Rational-isogen}, and over global function fields by various authors (see eg. \cite{Poonen2007Gonality-of-mod}), so we know that \ref{conjecture_ST} holds in our situation, hence $\tor(1)$ is not Zariski dense in $S$. 
\end{proof}

In general, \ref{conjecture_SSP} might be expected to follow from \ref{conjecture_ST} and very good lower bounds on the heights of non-torsion points. The remainder of this article will be devoted to proving \ref{conjecture_SSP} in some special cases for abelian varieties of higher dimension.  

\subsection{Strategy for proving more special cases of \ref{thm:main_sparsity}}

The remainder of the paper will be devoted to proving \ref{conjecture_SSP} for families of curves over `geometrically simple' base schemes and assuming that the jacobian admits a N\'eron model over that base (see \ref{thm:main_sparsity} for the precise statement). As discussed in the introduction we will use a number of tools from the theory of N\'eron models and height jumps developed in the papers \cite{Holmes2014Neron-models-an}, \cite{Holmes2014A-Neron-model-o}, \cite{David-Holmes2014Neron-models-an} and \cite{Ignacio-Burgos-Gil2015The-singulariti}. In an attempt to keep the current work reasonably self-contained we will briefly recall the main definitions and results we use as we go along. 

The strategy of the proof was briefly discussed in the introduction, but we will give a slightly more detailed outline here before proceeding. We begin with a smooth projective variety $S/K$ and a nodal curve $C/S$; by a \emph{nodal curve} we mean a proper flat finitely presented morphism all of whose geometric fibres are reduced, connected, of dimension 1, and have at worst ordinary double point singularities. Write $U\sub S$ for the locus where $C/S$ is smooth; we assume $U$ is dense in $S$. Write $J$ for the jacobian of $C_U/U$, and let $\sigma \in J(S)$ be a section of infinite order. We consider the function 
\begin{equation*}
\hat{\on{h}}_\sigma\colon U(\bar{\bb{Q}}) \ra \bb{R}_{\ge 0}; \;\; s \mapsto \hat{\on{h}}(\sigma(s)). 
\end{equation*}
We decompose $\hat{\on{h}}_\sigma$ as a sum of two functions, 
\begin{equation*}
\hat{\on{h}}_\sigma = \on{h}_\cl{L} + j
\end{equation*}
where $\on{h}_\cl{L}$ is a Weil height on $U$ with respect to a certain line bundle on $S$, and $j$ is an `error term', the \emph{height jump}. 

The line bundle $\cl{L}$ will be the admissible extension of Deligne pairing of $\sigma$ with itself, cf. \ref{sec:defining_height_jump}. In \ref{sec:defining_height_jump} we will also define the height jump $j$. In \ref{sec:computing_jump} we will make the jump explicit in certain situations, following closely the presentation in \cite{David-Holmes2014Neron-models-an}. This will require a discussion of Green's functions on resistive networks. In \ref{sec:aligned_jump_vanishes} we will show that the jump vanishes if and only if $J$ admit a N\'eron model over a suitable $\bb{Z}$-model of $S$; this is not essential for our other results, but demonstrates the close link between N\'eron models and the height jump. In \ref{sec:gen_aligned_bound_j} we return to our main aim of proving \ref{thm:main_sparsity}, showing that the jump is bounded for curves the jacobians of whose generic fibres admit N\'eron models. In \ref{sec:heights_and_rat} we compare ways of associating heights to line bundles which are not ample, giving a key inequality needed for the proof of \ref{thm:main_sparsity}. Finally in \ref{sec:proof_of_second_main} we put these ingredients together to prove our main result \ref{thm:main_sparsity}.

\subsection{Defining the algebraic height jump}\label{sec:defining_height_jump}

Let $S$ be a regular scheme, $C \ra S$ a generically smooth nodal curve, and $U \sub S$ the largest open over which $C$ is smooth. Write $J$ for the jacobian of $C_U /U$. Let $\sigma$, $\tau \in J(U)$ be two sections. Write $\ca{P}$ for the rigidified Poincar\'e bundle on $J \times_U J$ (cf. \cite{Moret-Bailly1985Metriques-permi}). Write $\Span{\sigma, \tau}$ for the dual of the pullback $(\sigma, \tau)^*\ca{P}$ - it is a line bundle on $U$ called the \emph{Deligne pairing} of $\sigma $ and $\tau$, cf. \cite{David-Holmes2014Neron-models-an}. 

We know from \cite{Holmes2014Neron-models-an} that there exists a largest open subset $U \sub V \sub S$ such that the complement of $V$ in $S$ has codimension 2 and such that $J$ admits a N\'eron model over $V$. The N\'eron model is of finite type by \cite{Holmes2016Quasi-compactne}, so there exists $n>0$ such that $n\sigma$ and $n\tau$ pass through the identity component $N^0$ of the N\'eron model. By \cite[definition II.1.2.7 and theorem II.3.6]{Moret-Bailly1985Pinceaux-de-var} the Poincar\'e bundle admits a unique rigidified extension to $N^0\times_S N^0$, which we pull back to $V$ along $(n\sigma,n\tau)$. Since $S$ is regular and the complement of $V$ has codimension 2, this line bundle on $V$ has a unique extension to a line bundle on $S$. Raising this line bundle to the power $1/n^2$ we obtain a $\bb{Q}$-line bundle on $S$ which we call the \emph{admissible extension of the Deligne pairing} (or just \emph{the admissible pairing}), and write $\Span{\sigma, \tau}_a$ - it is independent of the choice of $n$. This is also known as the `Lear extension', as in the setting of complex geometry (or more generally Hodge theory) it can also be defined by requiring that certain metrics extend continuously outside some codimension 2 subset of the boundary, see for example \cite{Lear1990Extensions-of-n}, \cite{Hain2013Normal-function}, \cite{Ignacio-Burgos-Gil2015The-singulariti}.

Suppose $T$ is another regular scheme, and $f\colon T\ra S$ is a morphism. We say $f$ is \emph{non-degenerate} if $f^{-1}U$ is dense in $T$. Given a non-degenerate morphism $f\colon T \ra S$, we have a canonical isomorphism of line bundles on $f^{-1}U$
\begin{equation*}
f^*\Span{\sigma, \tau}|_{f^{-1}U} = \Span{f^*\sigma, f^*\tau}|_{f^{-1}U} = \Span{f^*\sigma, f^*\tau}_a|_{f^{-1}U}. 
\end{equation*}
Thus the bundle
\begin{equation*}
f^*\Span{\sigma, \tau}_a^\vee \otimes \Span{f^*\sigma, f^*\tau}_a
\end{equation*}
is canonically trivial over $f^{-1}U$, and so the section `$1$' over $U$ gives a canonical rational section over $T$. We define the \emph{height jump} $j(\sigma, \tau, f)$ associated to $\sigma$, $\tau$ and $f$ to be the corresponding $\bb{Q}$-divisor on $T$.

\begin{lemma} \label{mostlytrivial}
Assume $f \colon T \ra  S$ is a non-degenerate morphism such that $f^{-1}(S\setminus V)$ has codimension at least 2 in $T$. Then the formation of $\Span{D,E}_a$ is compatible with base change along $f$.
\end{lemma}
\begin{proof} This is a slight generalisation of \cite[proposition 3.4]{David-Holmes2014Neron-models-an}; now we know the N\'eron model is of finite type, the same proof will work with trivial modifications. 
\end{proof}

\begin{corollary} \label{cor:NM_implies_no_jump}Assume that $J_U$ extends into a N\'eron model over $S$. Then for all non-degenerate morphisms $f \colon T \ra S$, the height jump divisor on $T$ is trivial.
\end{corollary}
In \ref{prop:equivalence} we shall see that the sufficient condition of this corollary is also essentially a necessary condition, even if we restrict our test objects to traits. 

\begin{remark}
 The jump is defined for arbitrary $T$, but it is generally enough to be able to compute it when $T$ is a trait, since the jump is stable under flat base-change by \ref{mostlytrivial}, in particular under localisation at generic points of prime divisors in $T$. 

\end{remark}

\subsection{Computing the algebraic height jump}\label{sec:computing_jump}

In this section we summarise some results on the height jump from \cite{David-Holmes2014Neron-models-an}. In particular, \ref{formula_jump_green} shows how to compute the jump explicitly in certain situations, which will be important for our later results. We will relate the height jump to Green's functions on graphs. For some of this we will use the language of electrical networks, but we only use it as a language and as a guide to intuition; all our proofs are still (intended to be) rigorous. More details of what follows in the remainder of this section can be found in \cite[\S6 and \S7]{David-Holmes2014Neron-models-an}. 

\subsubsection{Resistive networks and Green's functions}
For us graphs are allowed loops and multiple edges, and are not directed. A \emph{resistive network} is a graph $\Gamma$ together with a labelling $\mu\colon \on{Edges}(\Gamma) \ra \bb{R}_{\ge 0}$ of the edges by non-negative real numbers, which we can think of as resistances (if the labels are all positive, it is called a \emph{proper} resistive network, otherwise it is \emph{improper}, which is an important distinction in what follows). 

The \emph{Laplacian} of a proper resistive network $(\Gamma, \mu)$ is the linear map $L\colon \bb{R}^{\on{Vert} \Gamma} \ra \bb{R}^{\on{Vert} \Gamma}$ which sends a vector $(x_v:v \in \on{Vert}\Gamma)$ to the vector whose component at a vertex $u$ is 
\begin{equation*}
\sum_{v \in \on{Vert}\Gamma} \sum_{e} \frac{x_{u}-x_v}{\mu(e)}
\end{equation*}
where the second sum is over all edges with one endpoint at $v$ and the other at $u$. 
\begin{remark}
If we think of a vector $X = (x_v:v \in \on{Vert}\Gamma)$ as an assignment of a voltage to every vertex in $\Gamma$, then for a directed edge $e$ from $u$ to $v$, the term $(u-v)/e$ can be thought of as the current flowing along the edge $e$. The component of the vector $LX$ at a vertex $v$ is then the total current flowing out of the vertex $v$ into the rest of the network. 
\end{remark}
\begin{definition}
Assume that $(\Gamma, \mu)$ is a proper resistive network with exactly one connected component. Let $X$, $Y \in \bb{R}^{\on{Vert}\Gamma}$. We define the \emph{Green's function} at $X$, $Y$ to be 
\begin{equation*}
\on{gr}(\Gamma, \mu;X, Y) = XL^+Y, 
\end{equation*}
where $L^+$ is the Moore-Penrose pseudo inverse of $L$. 
\end{definition}
If we fix $X$ and $Y$ and allow $\mu$ to vary over $\bb{R}_{>0}^{\on{Edges}\Gamma}$ then it is easy to see that the function sending $\mu$ to $\on{gr}(\Gamma, \mu;X, Y)$ is continuous. We can also extend the definition of the Green's function to improper networks:
\begin{definition}
Let $(\Gamma, \mu)$ be a resistive network with exactly one connected component and let $X$, $Y \in \bb{R}^{\on{Vert}\Gamma}$. Let $\Gamma'$ be the graph obtained from $\Gamma$ by contracting every edge with resistance 0, and write $\mu'$, $X'$ and $Y'$ for the corresponding resistance function and vertex weightings on $\Gamma'$. Then we define the \emph{Green's function} at $X$, $Y$ to be 
\begin{equation*}
\on{gr}(\Gamma, \mu;X, Y) = X'L'^+Y', 
\end{equation*}
where $L'^+$ is the Moore-Penrose pseudo inverse of the Laplacian $L'$ on $\Gamma'$. 
\end{definition}
Now for fixed $X$ and $Y$ we get a function $\bb{R}_{\ge0}^{\on{Edges}\Gamma} \ra \bb{R}$ sending $\mu$ to $\on{gr}(\Gamma, \mu;X, Y)$. It is no longer obvious that this function should be continuous, but it is in fact continuous, see \cite[proposition 6.6]{David-Holmes2014Neron-models-an}. 

\subsubsection{Some terminology for families of nodal curves}\label{sec:NM_finite_type}
In what follows we will need some precise descriptions of the local structure of families of nodal curves, which we will collect here. Let $S$ be a scheme, and $C/S$ a nodal curve. If $s\in S$ is a point then a \emph{non-degenerate trait through $s$} is a morphism $f\colon T \ra S$ from the spectrum $T$ of a discrete valuation ring, sending the closed point of $T$ to $s$, and such that $f^*C$ is smooth over the generic point of $T$. 

We say a nodal curve $C/S$ is \emph{quasisplit} if the morphism $\on{Sing}(C/S) \ra S$ is an immersion Zariski-locally on the source (for example, a disjoint union of closed immersions), and if for every field-valued fibre $C_k$ of $C/S$, every irreducible component of $C_k$ is geometrically irreducible. 

Suppose we are given $C/S$ a quasisplit nodal curve and $s \in S$ a point. Then we write $\Gamma_s$ for the \emph{dual graph} of $C/S$ - this makes sense because $C/S$ is quasisplit and so all the singular points are rational points, and all the irreducible components are geometrically irreducible. Assume that $C/S$ is smooth over a schematically dense open of $S$. If we are also given a non-smooth point $c$ in the fibre over $s$, then there exists an element $\alpha \in \ca{O}_{S,s}$ and an isomorphism of completed \'etale local rings (after choosing compatible geometric points lying over $c$ and $s$)
\begin{equation*}
\widehat{\ca{O}^{et}}_{C,c} \stackrel{\sim}{\ra} \frac{\widehat{\ca{O}^{et}}_{S,s}[[x,y]]}{(xy-\alpha)}. 
\end{equation*}
This element $\alpha$ is not unique, but the ideal it generates in $\ca{O}_{S,s}$ is unique. We label the edge of the graph $\Gamma_s$ corresponding to $c$ with the ideal $\alpha\ca{O}_{S,s}$. In this way the edges of $\Gamma_s$ can be labelled by principal ideals of $\ca{O}_{S,s}$. 

If $\eta$ is another point of $S$ with $s \in \overline{\{\eta\}}$ then we get a \emph{specialisation map} 
\begin{equation*}
\on{sp}\colon\Gamma_s \ra \Gamma_\eta
\end{equation*}
on the dual graphs, which contracts exactly those edges in $\Gamma_s$ whose labels generate the unit ideal in $\ca{O}_{S, \eta}$. If an edge $e$ of $\Gamma_s$ has label $\ell$, then the label on the corresponding edge of $\Gamma_\eta$ is given by $\ell\ca{O}_{S,\eta}$.

\subsubsection{The jump in terms of Green's functions}
Now we will apply our discussion of Green's functions above to computing the height jump. Let $S$ be a regular noetherian scheme, $C/S$ a generically-smooth quasisplit nodal curve, and $U\sub S$ the largest open over which $C$ is smooth. Let $s \in S$ be a point, and write $(\Gamma, \ell)$ for the labelled graph of $C/S$ at $s$, where the labels take values in the monoid of principal ideals of $\ca{O}_{S,s}$, cf. \ref{sec:NM_finite_type}. Let $Z_1, \cdots, Z_r$ be prime divisors in $S$ forming the boundary of $U$, and for each $i$ let $z_i$ be a local equation of $Z_i$ in the local ring $\ca{O}_{S,s}$. Now if $e$ is an edge of $\Gamma$, the label $\ell(e)$ is generated by $z_1^{a_1}\cdots z_r^{a_r}$ for some unique $(a_1, \ldots, a_r) \in \bb{Z}_{\ge 0}^r$. Given $I \sub \{1 , \ldots,  r\}$ we define a new labelled graph $(\Gamma, \ell_I)$ by requiring that, if $\ell(e) = (z_1^{a_1}\cdots z_r^{a_r})$, then $\ell_I(e) = (\prod_{i \in I}z_i^{a_i})$. In particular, $\ell_{\{1, \ldots, r\}} = \ell$. We write $\ell_i$ for $\ell_{\{i\}}$. 

Given a horizontal divisor $D$ on $C/S$ supported on sections through the smooth locus and of relative degree 0, we can define an associated combinatorial divisor $\ca{D}$ in $\bb{R}^{\on{Vert}\Gamma}$ by setting
\begin{equation*}
\ca{D}(y) = \on{deg}D|_Y
\end{equation*}
for each vertex $y$ of $\Gamma$ with corresponding irreducible component $Y$ of $C_s$. 

Suppose now that we have a non-degenerate trait $f\colon T \ra S$ through $s$. Then we have a pullback map $f^\# \colon \ca{O}_{S,s} \ra \ca{O}_T(T)$. We obtain a labelling of the edges of $\Gamma$ with values in $\bb{Z}_{\ge 0}$ by sending an edge $e$ to $\on{ord}_Tf^\#\ell(e)$, and we write $\on{ord}_Tf^\#\ell$ for this labelling. We make corresponding definitions for the $\ell_I$. 

\begin{theorem}[\cite{David-Holmes2014Neron-models-an}, theorem 7.8]
Suppose we are given two horizontal divisors $D$ and $E$ on $C/S$ of relative degree 0 and supported on sections through the the smooth locus of $C/S$, with associated combinatorial divisors $\ca{D}$ and $\ca{E}$. Write $\sigma$ (resp. $\tau$) for the section of the jacobian $J$ of $C_U$ induced by $D$ (resp. $E$). Then the height jump associated to $\sigma$, $\tau$ and $f$ is given by the formula $j(\sigma, \tau, f) = j\cdot [t]$ where $t$ is the closed point of $T$, and $j$ is given by
\begin{equation}\label{formula_jump_green}
j = \on{gr}(\Gamma, \on{ord}_Tf^\#\ell;\ca{D}, \ca{E}) - \sum_{i=1}^r\on{gr}(\Gamma, \on{ord}_Tf^\#\ell_i;\ca{D}, \ca{E}). 
\end{equation}
\end{theorem}

\subsection{Vanishing of the jump is equivalent to the existence of a N\'eron model}\label{sec:aligned_jump_vanishes}

Let $S$ be an integral, noetherian, regular scheme, and let $C \to S$ be a generically smooth nodal curve. Let $U \subset S$ be the largest open subscheme of $S$ over which $C$ is smooth. We have seen that if the jacobian of the generic fibre of $C \to S$ has a N\'eron model over $S$, then for each non-degenerate morphism $f \colon T \to S$ the height jump divisor is trivial. In this section we give a partial converse to this result, using \ref{formula_jump_green}. This result is not needed to prove our main \ref{thm:main_sparsity}, but is intended to demonstrate the close connection between the vanishing of the jump and the existence of N\'eron models. 

We recall the following definition of alignment of labelled graphs and curves from \cite{Holmes2014Neron-models-an}: 
\begin{enumerate}
\item given a graph $\Gamma$ with an edge-labelling by a (multiplicatively written) monoid, we say $\Gamma$ is \emph{aligned} if for all cycles $\gamma$ in $\Gamma$, and for all pairs of edges $e_1$, $e_2$ on $\gamma$, there exist positive integers $n_1$, $n_2$ such that
\begin{equation*}
(\on{label}(e_1))^{n_1} = (\on{label}(e_2))^{n_2} \, ; 
\end{equation*}
\item given a geometric point $s \in S$, we say $C \to S $ is \emph{aligned at $s$} if the labelled reduction graph $\Gamma_s$ of $C$ above $s$ is aligned (where the monoid is the monoid of principal ideals in $\ca{O}_{S,s}$); 
\item we say \emph{$C \to S$ is aligned} if $C \to S$ is aligned at $s$ for all geometric points $s$ in $S$. 
\end{enumerate}
\begin{proposition}\label{prop:equivalence}  Assume $C$ admits a regular nodal model over $S$. The following conditions are equivalent:
\begin{enumerate}[label=$(\alph*)$]
\item \label{item:NM} The jacobian of $C_U \ra U$ has a N\'eron model over $S$;
\item \label{item:no_jump} For all nodal models $\tilde{C} \ra S$ of $C$, for all non-degenerate morphisms $f \colon T \to S$ with $T$ a trait and for all relative degree zero divisors $D, E$ with support contained in the smooth locus $\mathrm{Sm}(\tilde{C}/S)$ of $\tilde{C} \to S$, the height jump divisor $J(f;D,E)$ is trivial;
\item \label{item:exists_no_jump} There exists a nodal model $\tilde{C} \ra S$ of $C$ with the same property as in \ref{item:no_jump}; 
\item \label{item:exists_reg_aligned} There exists an aligned regular model of $C \to S$;
\item \label{item:all_aligned}Every nodal model of $C$ over $S$ is aligned. 
\end{enumerate}
\end{proposition}
We note that following \cite[Proposition 3.6]{Jong1996Smoothness-semi}, a sufficient condition for the existence of a regular nodal model of $C$ is that $C \to S$ be split nodal (in the sense of \cite{Jong1996Smoothness-semi}) and smooth over the complement of a strict normal crossings divisor in $S$. 
\begin{proof} 
The equivalence of \oref{item:NM}, \oref{item:exists_reg_aligned} and \oref{item:all_aligned} follows from the main results of \cite{Holmes2014Neron-models-an}. It is clear that 
\oref{item:no_jump} implies \oref{item:exists_no_jump}. The implication \oref{item:NM} $\implies$ \oref{item:no_jump} follows from \ref{mostlytrivial}. We will now show \oref{item:exists_no_jump} implies \oref{item:exists_reg_aligned}. In fact we will show the contrapositive, that ($\neg$\oref{item:exists_reg_aligned}) $\implies$ ($\neg$\oref{item:exists_no_jump}). Assume ($\neg$ \oref{item:exists_reg_aligned}), and consider any regular nodal model of $C$ over $S$. To simplify notation, we will assume this model is $C$ itself. By ($\neg$ \oref{item:exists_reg_aligned}) we have that $C \to S$ is not aligned. Both questions are \'etale-local on the base, and so we may assume that $S$ is the spectrum of a strictly henselian local ring, with closed point $s$, and that $C \to S$ is not aligned at $s$. Note that $C \to S$ is quasisplit since $S$ is strictly henselian. Let $z_1, \ldots, z_r$ be distinct height 1 prime ideals in $O_{S,s}$ such that $C$ is smooth over $U := S \setminus V(\prod_{i=1}^r z_i)$. Write $\Gamma_s$ for the labelled graph of $C$ over $s$. Since $C$ is regular, every label on $\Gamma_s$ is equal to one of the $z_i$; non-trivial powers and products cannot occur. 

Let $\gamma$ be a cycle in $\Gamma_s$ which contains two distinct labels - such a cycle exists exactly because $C \to S$ is not aligned at $s$. After possibly reordering the $z_i$, we may and do assume that $\gamma$ contains two adjacent edges $e_1$, $e_2$ with labels $z_1$, $z_2$ respectively. Let $p$, $q \in C(S)$ be sections through the smooth loci of the two irreducible components of $C_s$ corresponding to the (distinct) endpoints of $e_1$. Let $D = E = p-q$, degree-zero divisors on $C \to S$ supported on the smooth locus. Applying \ref{formula_jump_green}, we find that for any non-degenerate test curve $f\colon T \ra S$ the height jump is given by 
\begin{equation*}
g(\Gamma_s,\on{ord}_t f^\# \ell_s;\ca{D},\ca{E}) - \sum_{i=1}^r g(\Gamma_s,\on{ord}_t f^\# \ell_{s,i};\ca{D},\ca{E}) \, . 
\end{equation*}
We find that all the terms $ g(\Gamma_s,\on{ord}_t f^\# \ell_{s,i};\ca{D},\ca{E})$ vanish for $i \neq 1$, since for $i \neq 1$, in terms of electrical resistances, all edges labelled $z_1$ are contracted (assume value zero) in the graphs $(\Gamma_s,\on{ord}_t f^\# \ell_{s,i})$, so the resistance between $p$ and $q$ becomes zero. To prove ($\neg$ \oref{item:exists_no_jump}), it suffices to show that the equality
\begin{equation}\label{eq:must_fail}
g(\Gamma_s,\on{ord}_t f^\# \ell_s;\ca{D},\ca{E}) = g(\Gamma_s,\on{ord}_t f^\# \ell_{s,1};\ca{D},\ca{E})
\end{equation}
\emph{fails} for some non-degenerate test curve $f \colon T \ra S$. Choose any non-degenerate test curve $f\colon T \ra S$ such that the image of the closed point $t$ of $T$ is the closed point $s$ of $S$. Then we find that every label 
of $(\Gamma_s,\on{ord}_t f^\# \ell_s)$ is strictly positive. Equation \oref{eq:must_fail} is equivalent to the statement that the effective resistance between $p$ and $q$ on the graph $\Gamma_s$ with edges given resistance equal to the $\on{ord}_t f^\# \ell_s$ is equal to the resistance between $p$ and $q$ on the same graph but with all the edges \emph{not} labelled by $z_1$ being contracted (equivalently, having their resistances set to zero). Then by \cite[Corollary A.5]{David-Holmes2014Neron-models-an}, in the graph with `non-contracted' edges, no current flows through any edge not labelled by $z_1$ when 1 unit of current flows in at $p$ and out at $q$. But since there is a path (extendable to a spanning tree) from $p$ to $q$ that either starts or ends with $e_2$, \cite[equation A.2]{David-Holmes2014Neron-models-an} tells us that the current along the $z_2$-labelled edge $e_2$ must be nonzero, so we have reached a contradiction. 
\end{proof}
\begin{remark} We point out that the regularity assumption in part \oref{item:exists_reg_aligned} of the theorem is essential: a nodal curve over $S = \on{Spec}\bb{C}[[u,v]]$ with labelled graph over the closed point a 1-gon with label $(uv)$ - which is then necessarily not regular - is aligned, but its jacobian does not admit a N\'eron model over $S$.  
\end{remark}

\subsection{Bounding the height jump for generically aligned curves}\label{sec:gen_aligned_bound_j}

Recall from \ref{prop:equivalence} that the height jump vanishes whenever the jacobian of a family of curves admits a N\'eron model. The aim of this section is to show that if the base-change of our family to $K$ admits a N\'eron model, then the height jump is bounded. Here we write $\Lambda$ for the spectrum of the ring of integers of $K$ if $K$ is a number field, or for a smooth proper geometrically integral curve over a finite field with field of rational functions $K$ if $K$ is a global function field. If $L/K$ is a finite extension then we define $\Lambda_L$ similarly. 

%

%

\begin{theorem}\label{thm:generic_alignment_bounds_jump}
Let $S$ be a regular scheme proper and flat over $\Lambda$, and let $C \ra S$ be a nodal curve smooth over a dense open subscheme $U \hra S$ and with $C$ regular. Write $J $ for the jacobian of $C_U/U$. Suppose we are given integers $d_1, \ldots, d_n, e_1, \ldots, e_n$ with $\sum_i d_i = \sum_i e_i = 0$ and sections $\sigma_1, \ldots, \sigma_n, \tau_1, \ldots, \tau_n \in C^{sm}(S)$, and set $\sigma = [\sum_i d_i(\sigma_i)_U]$, $\tau = [\sum_i e_i (\tau_i)_U] \in J(U)$.

Suppose also that $J_K$ has a N\'eron model over $S_K$. Then there exists a constant $B$ such that for all finite extensions $L/K$ and for all $x \in U(L)$ we have
\begin{equation*}
\left|\frac{j(\sigma, \tau, \bar{x})}{[L:K]}\right| \le B
\end{equation*}
where $j(\sigma, \tau, \bar{x})$ denotes the height jump associated to $\sigma$, $\tau$ and $\bar{x}\colon\Lambda_L \ra S$. 
\end{theorem}

\begin{proof}
We proceed in several steps. 
\begin{enumerate}
\item[Step 1:] translate the problem into a statement which is local on $\Lambda$. 

Recall from \cite{Holmes2014Neron-models-an} that the existence of a N\'eron model of the jacobian is equivalent to alignment of the family of curves. Using that $C_K/S_K$ is aligned and the finite presentation of $C/S$, we see that there is a nonempty Zariski open $\Lambda^0 \hra \Lambda$ such that the jacobian of $C_{U_{\Lambda^0}}/U_{\Lambda^0}$ admits a N\'eron model over $S_{\Lambda^0}$. Then by \ref{cor:NM_implies_no_jump} we deduce that there is a finite set of primes of $\Lambda$ such that for every (finite flat quasi-)section, the height jump vanishes outside that finite set. Because the claim is stable under unramified base-change on $\Lambda$ we may replace $\Lambda$ by its strict henselisation and completion $\Lambda_\frak{p}$ at some prime $\frak{p}$. Write $K_\frak{p}$ for the field of rational functions on $\Lambda_\frak{p}$. 

\item[Step 2:] translate the problem into a statement which is local on $S$. 

Let $S' \ra S$ be an \'etale cover, and let $L/K_\frak{p}$ be a finite extension. Write $\Lambda_L$ for the normalisation of $\Lambda_\frak{p}$ in $L$. Then $\Lambda_L$ is finite over $\Lambda_\frak{p}$ of degree $[L:K]$, and is strictly henselian. By properness of $S/\Lambda$ the restriction map $S(\Lambda_L) \ra S(L)$ is a bijection, and since $\Lambda_L$ is strictly henselian we know that $S'(\Lambda_L) \ra S(\Lambda_L)$ is surjective. By quasi-compactness of $S$, it therefore suffices to bound the height jump on every connected component of some \'etale cover of $S$. 


\item[Step 3:] make some reductions using the \'etale local nature of the statement.

To keep the notation concise we will write $\Lambda$ in place  of $\Lambda_\frak{p}$, and (using step 2 and \cite[lemma 6.3]{Holmes2014A-Neron-model-o}) we will replace $S$ by a flat finite-type integral $\Lambda$-scheme such that for some $s \in S$ (which we will refer to as a `controlling point'), the induced specialisation map $\Gamma_s \ra \Gamma_t$ is surjective for all points $t \in S$.



Since $S$ is noetherian, it suffices to consider quasi-sections in $S(\Lambda_L)$ (with $L/K$ finite) such that the induced graph over the closed point of $\Lambda$ is the same as the graph of the controlling point $s$. If $Z_1, \ldots, Z_n$ are prime Weil divisors on $S$ such that every label on $\Gamma$ can be written in terms of the $Z_i$ (and the $Z_i$ are minimal with respect to this) then we can restrict further to quasi-sections $x\in S(\Lambda_L)$ such that for all $i$, the divisor $x^*Z_i$ on $\Lambda$ is non-trivial. Write $\Sigma$ for the set of such quasi-sections. If $x\colon \Lambda_L \ra S$ is a quasi-section, we write $L = L_x$ and $\Lambda_x = \Lambda_L$. 

\item[Step 4:] conclude the argument by an appeal to a theorem about resistive networks. 

We are now in the situation of \ref{sec:computing_jump}, and we adopt the notation of that section. Further, we know that $C_K/S_K$ is aligned, where we write $K$ for the generic point of $\Lambda$. Some of the $Z_i$ may be trivial after pullback to $S_K$, and some may not. Re-ordering we assume that $Z_1, \cdots, Z_r$ are not trivial on $S_K$ and $Z_{r+1}, \cdots, Z_n$ are. As such, for each $r+1 \le i\le n$ the function 
\begin{equation*}
\Sigma \ra \bb{R}_{>0}; \;\; x \mapsto \frac{\on{ord}_{\Lambda_x} x^*Z_i}{[L_x:K]}
\end{equation*}
is bounded. 

Define divisors on $C/S$ by $D = \sum_i d_i\sigma_i$, $E = \sum_i e_i \tau_i$, so $\sigma = [D_U]$ and $\tau = [E_U]$. Write $\ca{D}$, $\ca{E}$ for the combinatorial divisors associated to $D$ and $E$. Then applying \cite[proposition 6.8]{David-Holmes2014Neron-models-an} we find that 
\begin{equation*}
\left|
\on{gr}\left( \Gamma, \frac{\on{ord}_{\Lambda_x} x^\#\ell_{}}{[L_x:K]}; \ca{D}, \ca{E}  \right)
-
\on{gr}\left( \Gamma, \frac{\on{ord}_{\Lambda_x} x^\#\ell_{\{1, \ldots, r\}}}{[L_x:K]}; \ca{D}, \ca{E}  \right)
- 
\sum_{i=1}^r \on{gr}\left( \Gamma, \frac{\on{ord}_{\Lambda_x} x^\#\ell_{i}}{[L_x:K]}; \ca{D}, \ca{E}  \right)
\right|
\end{equation*}
is bounded independent of $x\in\Sigma$. To complete the proof it suffices to bound the absolute value of 
\begin{equation}\label{eq:to_bound}
B(x)\coloneqq \on{gr}\left( \Gamma, \frac{\on{ord}_{\Lambda_x} x^\#\ell_{\{1, \ldots, r\}}}{[L_x:K]}; \ca{D}, \ca{E}  \right)
- 
\sum_{i=1}^r \on{gr}\left( \Gamma, \frac{\on{ord}_{\Lambda_x} x^\#\ell_{i}}{[L_x:K]}; \ca{D}, \ca{E}  \right)
\end{equation}
for which we will use that $C_K/S_K$ is aligned. In fact, we will show that $B(x) = 0$ for all $x \in \Sigma$. 


Let $\Gamma_K$ be the graph obtained from $\Gamma$ by contracting every edge whose label does not contain at least one of $Z_1, \cdots, Z_r$ with non-zero (i.e. positive) coefficient. It is clear that all terms in \ref{eq:to_bound} can be computed on $\Gamma_K$ just as well as on $\Gamma$. Moreover, $\Gamma_K$ is the graph of $C$ over a controlling point of $S_K$, so we can see easily what alignment of $C_K/S_K$ means on $\Gamma_K$; namely, that all labels on edges in 2-vertex-connected components are multiplicatively related over $S_K$. Define a labelling $\ell'$ on the edges of $\Gamma_K$ by composing $\ell$ with the map sending $\sum_{i=1}^n a_i Z_i$ to $\sum_{i=1}^r a_i Z_i$. 

We will show $B(x) = 0$ in three steps:
\begin{itemize}
\item[Step 4.1:] the case where $\Gamma_K$ is 2-vertex connected. 

By alignment there exists a divisor $\delta =\sum_{i=1}^r a_i Z_i$ and for each edge $e$ of $\Gamma_K$ a constant $\lambda_e \in \bb{Q}_{\ge 0}$ such that $\ell'(e) = \lambda_e \delta $. Let $\rho\colon \bb{R}_{\ge 0}^{\on{Edges}\Gamma_K} \ra \bb{R}_{\ge 0}$ be the map sending a labelling $\mu$ to the Green's function $\on{gr}(\Gamma_K, \mu;\ca{D},\ca{E})$. This function is homogenous of degree 1 by \cite[proposition 6.7(a)]{David-Holmes2014Neron-models-an}. For $1 \le i \le r$ define 
\begin{equation*}
g_i\colon \Sigma \ra \bb{R}_{\ge 0}^{\on{Edges}\Gamma_K}; \;\; x \mapsto \frac{\on{ord}_{\Lambda_x} x^\#\ell_i}{[L_x:K]} 
\end{equation*}
and define $g = \sum_{i=1}^rg_i$. Then by definition we have
\begin{equation*}
\on{gr}\left( \Gamma_K, \frac{\on{ord}_{\Lambda_x} x^\#\ell_{\{1, \ldots, r\}}}{[L_x:K]}; \ca{D}, \ca{E}  \right) = \rho \circ g
\end{equation*}
and 
\begin{equation*}
\on{gr}\left( \Gamma_K, \frac{\on{ord}_{\Lambda_x} x^\#\ell_{i}}{[L_x:K]}; \ca{D}, \ca{E}  \right) = \rho \circ g_i
\end{equation*}
for $1 \le i \le r$. Define
\begin{equation*}
f\colon \bb{R}_{\ge 0} \ra \bb{R}_{\ge 0}^{\on{Edges}\Gamma_K}; \;\; t \mapsto (\lambda_e t)_{e \in \on{Edges}\Gamma_K}
\end{equation*}
and define $\rho_0 \coloneqq \rho \circ f\colon \bb{R}_{\ge 0} \ra \bb{R}_{\ge 0}$, which is a `linear' map since $\rho$ is homogeneous of weight 1 and the source has dimension 1. Setting 
\begin{equation*}
h_i\colon \Sigma \ra \bb{R}_{\ge 0}; \;\; x \mapsto a_i \frac{\on{ord}_{\Lambda_x} x^\#Z_i}{[L_x:K]}
\end{equation*}
and $h = \sum_{i=1}^r h_i$, we find $g_i = f \circ h_i$ and $g = f \circ h$. Thus $\rho \circ g = \rho_0 \circ h$, and $\rho \circ g_i = \rho_0 \circ h_i$ for each $i$. Then 
\begin{equation*}
\begin{split}
B & = \rho_0 \circ h - \sum_{i=1}^r\rho_0 \circ h_i\\
& = \rho_0 \circ \left( h - \sum_{i=1}^r h_i\right)\;\;\;\;\;\;\;\;\;\;\;\text{(by linearity of }\rho_0)\\
&=  \rho_0 \circ (\text{zero map}) = (\text{zero map}). 
\end{split}
\end{equation*}

\item [Step 4.2:] the case where $\ca{D} = \ca{E} = u-v$ for some vertices $u$ and $v$. 

We may assume $\Gamma_K$ has no self-loops, since they do not contribute to the Green's functions. Write $\Gamma'$ for the graph obtained from $\Gamma_K$ by contracting to a point every 2-vertex-connected component which is not a single edge (so $\Gamma'$ is a tree). Thus there is a unique path $\gamma$ from the image of $u$ to the image of $v$ in $\Gamma'$. 

Let $H_1, \ldots, H_n$ be the vertices appearing along $\gamma$, so the image of $u$ lies in $H_1$ and the image of $u$ lies in $H_n$. Let $b_1, \ldots, b_{n-1}$ be the edges along $\gamma$. Each $H_i$ corresponds to either a vertex or a 2-vertex-connected subgraph of $\Gamma_K$, which we will denote by the same symbol $H_i$. For $1 \le i \le n-1$ let $t_i \in H_i$ be the vertex which the edge of $\gamma$ out of $H_i$ lifts to, and for $2 \le i \le n$ let $s_i\in H_i$ be the vertex of $H_i$ where the edge of $\gamma$ into $H_i$ lifts to. The result might look something like this (if $n=4$):

\includegraphics[scale=1.5]{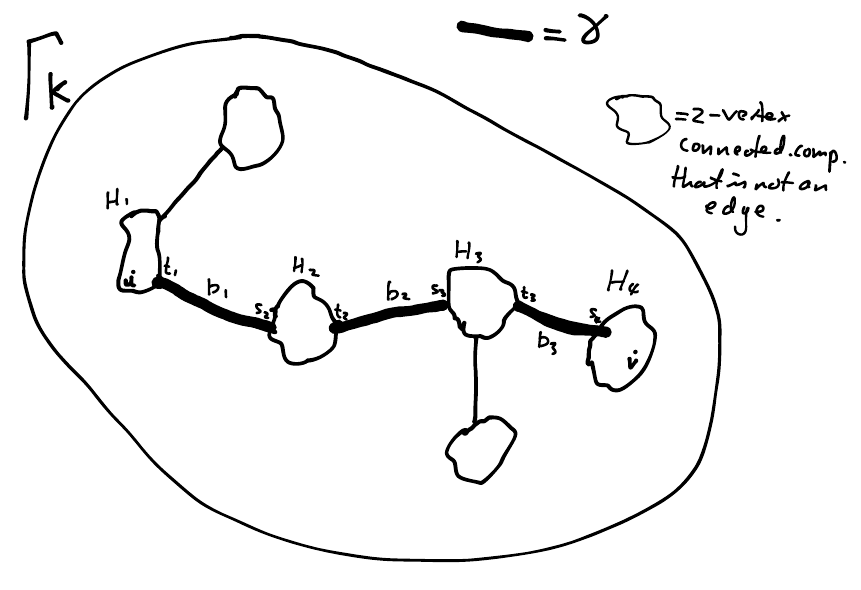}

Then by additivity of Green's functions in trees we find that 
\begin{equation}\label{eq:vanish_two_points}
B(x) = B_x(H_1; u, t_1) + \sum_{i=2}^{n-1}B_x(H_i; s_i, t_i) + \sum_{i=1}^{n-1}B_x(b_i; t_i, s_i) + B_x(H_n; t_n, v). 
\end{equation}
where for a subgraph $G\sub \Gamma_K$ and two vertices $a$ and $b$ we write
\begin{equation*}
\begin{split}
& B_x(G;a,b) \coloneqq\\
& \on{gr}\left( G, \frac{\on{ord}_{\Lambda_x} x^\#\ell_{\{1, \ldots, r\}}}{[L_x:K]}|_G; a-b, a-b  \right)
- 
\sum_{i=1}^r \on{gr}\left( G, \frac{\on{ord}_{\Lambda_x} x^\#\ell_{i}}{[L_x:K]}|_G; a-b, a-b  \right)
\end{split}
\end{equation*}
But all the terms on the right hand side of \ref{eq:vanish_two_points} vanish by step 4.1, so we are done. 

\item[Step 4.3:] the general case. 

To deduce the general case from the case where $\ca{D} = \ca{E} = u-v$, we fix $x$ and consider $B$ as a bilinear function on the space $\on{Div}$ of combinatorial divisors of degree zero on $\Gamma_K$. Then $B$ is positive semi-definite by \cite[corollary 6.7(b)]{David-Holmes2014Neron-models-an} and it vanishes on a basis of $\on{Div}$ by step 4.2, so by the Cauchy-Schwarz inequality we find that $B$ is zero. 
\end{itemize}
\end{enumerate}
\end{proof}

\subsection{Heights and rational maps}\label{sec:heights_and_rat}
Let $S/\qq$ be a projective scheme, and $\ca{L}$ on $S$ a line bundle. To this data we can attach a Weil height
\begin{equation}
\on{h}_\ca{L}\colon S(\bar{\qq}) \ra \bb{R}
\end{equation}
which is unique up to $O(1)$. For example, this can be done via Arakelov theory, or by writing $\ca{L}$ as a difference of very ample line bundles and applying the usual height machinery on a variety embedded in projective space. In either case a number of choices must be made, but the resulting heights all differ by bounded amounts. We say $\on{h}_\ca{L}$ (or $\ca{L}$) is \emph{weakly non-degenerate} if the sets of points of bounded height with residue fields of bounded degree are not Zariski dense in $S$. For example, Northcott's theorem tells us that if $\ca{L}$ is ample then sets of points of bounded height and degree for $\on{h}_\ca{L}$ are finite, so certainly not Zariski dense if $S$ has positive dimension. The main goal of this section is to show that if $\on{h}^0(S, \ca{L}) \ge 2$, then $\on{h}_\ca{L}$ is weakly non-degenerate. We will start by developing a bit of theory about heights associated to rational maps. Since we are ultimately interested in weak non-degeneracy, we will work up to $O(1)$ everywhere. 

In the above setup, let $V$ be a non-zero sub-$\qq$-vector space of $\on{H}^0(S, \ca{L})$. In the usual way we obtain a rational map 
\begin{equation*}
f_V\colon S \dashrightarrow \bb{P}(V)
\end{equation*}
(we think of $\bb{P}(V)$ as the space of rank-1 quotients of $V$, to avoid having duals everywhere). Let $U$ be an open subset of its domain of definition. By composing the map $f_V$ with a standard height on $\bb{P}(V)$ (say after choosing a basis of $V$) we get a height
\begin{equation*}
\on{h}_V\colon U(\bar{\qq}) \ra\bb{R}. 
\end{equation*}
For example, the reader will easily verify that if $V \sub V' \sub \on{H}^0(S, \ca{L})$ then $f_{V'}$ is defined on $U$ and we have $\on{h}_V \le \on{h}_{V'}$ on $U(\bar{\qq})$ (up to $O(1)$). It is slightly harder to compare $\on{h}_{\ca{L}}$ with $\on{h}_V$, but we have
\begin{lemma}\label{lem:rational_map_height_inequality}
On $U(\bar{\qq})$ we have $\on{h}_V \le \on{h}_\ca{L} $ up to $O(1)$. 
\end{lemma}
Note that $\on{h}_\ca{L}$ is not in general equal to $\on{h}_{\ca{L}(S)}$. 
\begin{proof}
Let $\ca{S}$ be a proper flat reduced model of $S$ over $\ca{O}_K$ such that $\ca{L}$ extends to a line bundle on $\ca{S}$; choose such an extension and denote it $\cl{L}$. Let $\ca{L}_V$ be the sub-$\ca{O}_{S}$-module of $\ca{L}$ generated by $V$, and let $\cl{L}_V$ be an extension to a coherent submodule of $\cl{L}$. Let 
\begin{equation*}
I_V \coloneqq \cl{L}_V \otimes_{\ca{O}_\ca{S}}\cl{L}^\vee \ra \ca{O}_{\ca{S}}, 
\end{equation*}
 a coherent sheaf of ideals in $\ca{O}_{\ca{S}}$. Let $\pi\colon \tilde{\ca{S}} \ra \ca{S}$ be the blowup of $\ca{S}$ in $I_V$. For a point $p \in U(K)$, we write $\bar{p}$ for the corresponding $\ca{O}_K$-point of $\ca{S}$, and $\tilde{p}$ for the corresponding $\ca{O}_K$-point of $\tilde{\ca{S}}$. Write $f = f_V\colon U \ra \bb{P}(V)$, and define $\tilde{f} \colon \tilde{\ca{S}}_K \ra \bb{P}(V)$ to be the unique map extending $f$. Then we have
 \begin{equation*}
 \begin{split}
h(f(p)) & = h(\tilde{f}(p))\\
& \stackrel{(1)}{=}\widehat{\on{deg}} \left(\tilde{p}^*(\pi^{-1}I_v \otimes_{\ca{O}_\ca{S}}\cl{L})\right)\\
& \stackrel{(2)}{=}\widehat{\on{deg}} \left(\bar{p}^*(I_v \otimes_{\ca{O}_\ca{S}}\cl{L})\right)\\
& \stackrel{(3)}{\le}\widehat{\on{deg}} \left(\bar{p}^*\cl{L}\right). \\
 \end{split}
 \end{equation*}
 See for example \cite[\S 1.2]{Szpiro1985Degres-intersec} or \cite[IV, \S 3]{Lang1988Introduction-to} for the definition of the arithmetic degree $\widehat{\on{deg}}$; note that in the function field case it is just the usual degree of a line bundle on a proper smooth curve. 
 We now justify the above (in)equalities:
 \begin{enumerate}
 \item[(1)] holds up to $O(1)$; it follows from the fact that on $\tilde{\ca{S}}_K$ we have 
$\left(\pi^{-1}I_v \otimes \pi^*\cl{L}\right)|_{\tilde{\ca{S}}_K} = \tilde{f}^*\ca{O}(1)$, and $h(\tilde{f}(p))$ is equal (up to $O(1)$) to the arithmetic degree of $\tilde{p}^*$ of any invertible model of $\tilde{f}^*\ca{O}(1)$; 
 \item[(2)] is where the content lies; because $p$ is in $U$ we know that $I_V$ pulls back along $\bar{p}$ to a submodule of the line bundle $\bar{p}^*\cl{L}$ which is invertible on the generic point of $\ca{O}_K$, and hence is itself a line bundle. Then from the proof of \cite[8.1.5]{Liu2002Algebraic-geome} we find that $\bar{p}^*I_V = \tilde{p}^*\pi^{-1}I_V$, from which the result follows;
 \item[(3)] is because $I_V \sub \ca{O}_{\ca{S}}$, so $\bar{p}^*I_v \otimes_{\ca{O}_\ca{S}}\cl{L} \sub \bar{p}^*\cl{L}$. 
 \end{enumerate}
 Now up to $O(1)$ we have that $h_V(p) = h(f(p))$ and $h_\ca{L}(p) = \widehat{\on{deg}} \left(\bar{p}^*\cl{L}\right)$ and we are done. 
 \end{proof}

\begin{corollary}\label{lem:very_effective_implies_weakly_non_degen}
Let $S/\qq$ be a connected projective scheme, and $\ca{L}$ on $S$ a line bundle with $h^0(S, \ca{L}^{\otimes n}) \ge 2$ for some $n >0$. Then $\ca{L}$ is weakly non-degenerate. 
%
\end{corollary}
\begin{proof}
We may assume $n=1$. In the above notation, let $V = \on{H}^0(S, \ca{L})$. Then we can find a dominant rational map $f_V\colon S \dashrightarrow \bb{P}^1$ which is defined on some dense open $U \sub S$. From the Northcott property for $\bb{P}^1$ we see that $\on{h}_V$ is weakly non-degenerate, and the result then follows from \ref{lem:rational_map_height_inequality}. 
\end{proof}

\subsection{Proof of the second main \ref{thm:main_sparsity}}
\label{sec:proof_of_second_main}

\begin{remark}\label{rem:bounded_differences}
In the proof of \ref{thm:main_sparsity} we want to talk about two metrics on a line bundle having `bounded difference'. If $X$ is a finite-type scheme over $\bb{C}$ and ${L}$ is a line bundle on $X$ with metrics $\aabs{-}_1$ and $\aabs{-}_2$, we say $\aabs{-}_1$ and $\aabs{-}_2$ have \emph{bounded difference} if there exists an open Zariski cover $U_i$ of $X$ and generating sections $\ell_i \in {L}(U_i)$ such that the function
\begin{equation*}
\Big|\aabs{\ell_i}_1 - \aabs{\ell_i}_2\Big|
\end{equation*}
is bounded on $U_i$ for every $i$. 

Now let $\ca{X}$ be a proper scheme over $\Lambda$ and $\ca{L}$ a line bundle on $\ca{X}$. Let $X = \ca{X}_\bb{C}$ and let $L = \ca{L}_\bb{C}$. To any metric $\aabs{-}$ on $L$ we associate as usual a height function 
\begin{equation*}
\on{h}_{\ca{L}, \aabs{-}}\colon \ca{X}(\bar{K}) \ra \bb{R}. 
\end{equation*}
If $\aabs{-}_1$ and $\aabs{-}_2$ are two metrics on $L$ having bounded difference in the above sense, then one checks without difficulty that the functions $\on{h}_{\ca{L}, \aabs{-}_1}$ and $\on{h}_{\ca{L}, \aabs{-}_2}$ have bounded difference. 
\end{remark}

\begin{definition}
Let $S/\qq$ be a smooth projective connected scheme of dimension $d$. We say a line bundle $L$ on $S$ is \emph{ample-positive} if for every ample $H$ on $S$, we have that $L\cdot H^{d-1} > 0$. 
\end{definition}

\begin{theorem}\label{thm:main_sparsity}
Let $\ca{S}$ be a regular scheme, projective and flat over $\Lambda$. Let $C/\ca{S}$ be a nodal curve, smooth over a dense open of $\ca{S}$ and with $C_{\ca{S}_K}$ regular. Assume:
\begin{enumerate}
\item $C_{\ca{S}_K}/\ca{S}_K$ is aligned;
\item
for every ample-positive $L$ on $\ca{S}_K$ there exists $n >0$ such that $h^0(\ca{S}_K, L^{\otimes n}) \ge 2$. 
\end{enumerate}
Let $S \sub \ca{S}_K$ be any open over which $C$ is smooth and write $J$ for the jacobian of $C_S/S$. Let $\sigma \in J(S)$ be a section of infinite order corresponding to a divisor supported on some sections of $C/\ca{S}_K$. Then given any $d \in \bb{Z}_{\ge 1}$ there exist $\epsilon \in\bb{R}_{>0}$ such that the set 
\begin{equation*}
{\tor}_\epsilon(d) = \{p\in S(\bar{\qq}) | [\kappa(p): \qq] \le d \text{ and }\hat{\on{h}}(\sigma(p)) \le \epsilon\}. 
\end{equation*}
is not Zariski dense in $S$. 
\end{theorem}
Examples of varieties for which the second condition of the theorem holds include curves and varieties $X$ with $\dim_\bb{Q}\on{Pic}(X) \otimes_\bb{Z} \bb{Q} = 1$. 

\begin{proof}\leavevmode
\begin{enumerate}
\item[Step 1:]preliminary reductions. 

We may assume $S$ is integral; write $L$ for its field of rational functions. Suppose first that the section $\sigma$ lies in the image of the $L/K$ trace of $J_K$ (cf. \cite{Conrad2006Chows-K/k-image}). Then after shrinking $S$ we may assume there is an abelian variety $B/K$ and a map $B_L \ra J_L$ such that $\sigma$ is the base-change to $L$ of some section $\tau \in B(K)$. The section $\tau$ cannot have finite order since anything which kills $\tau$ also kills $\sigma$, and if $\tau$ is of infinite order then the result is clear since the set of points of sufficiently small height is empty. We may thus assume that $\sigma$ does \emph{not} lie in the image of the $L/K$ trace of $J_K$. 

We have already assumed that $\sigma$ corresponds to a divisor supported on some sections of $C/S$ but we may further assume that these sections extend $C/\ca{S}$ and (for example by looking at case I in the proof of theorem 2.4 of \cite{Knudsen1983The-projectivitii}) are contained in the smooth locus of $C/\ca{S}$. This can be achieved by modifications which do not affect the fibre over $K$, since the sections already necessarily go through the smooth locus over $K$ by our assumption that $C_{\ca{S}_K}$ be regular. 

\item[Step 2:] checking the admissible pairing has enough global sections. 

Define a line bundle on $\ca{S}$ by $\ca{L} = \Span{\sigma, \sigma}_a$, the admissible pairing as defined in \ref{sec:defining_height_jump}. We want to show that $\on{h}^0(\ca{S}_K, \ca{L}^{\otimes n}) \ge 2$ for some $n >0$. By \ref{lem:very_effective_implies_weakly_non_degen} it is enough to show that $\ca{L}$ is ample-positive, which we will deduce from the proof of the Lang-N\'eron theorem (of which a modern exposition can be found in \cite{Conrad2006Chows-K/k-image}). Let $H$ be any ample line bundle on $\ca{S}_K$, then the pair $(\ca{S}_K, H)$ gives rise to a \emph{generalised global field structure} on $L$, cf. [loc. cit., example 8.4].  After quite some work unravelling the definitions, we find that the N\'eron-Tate height 
\begin{equation*}
\hat{\on{h}}_H\colon J(\bar{L})/\on{Tr}(J/\bar{K}) \ra \bb{R}
\end{equation*}
differs by a bounded amount from the function 
\begin{equation*}
J(\bar{L})/\on{Tr}(J/\bar{K}) \ra \bb{R};\;\; \tau \mapsto c_1(\ca{L})\cdot c_1(H)^{\on{dim}\ca{S}_K - 1}. 
\end{equation*}
By [loc. cit. theorem 9.15] the function $\hat{\on{h}}_H$ is positive definite, so (after possibly replacing $\sigma$ by a positive multiple, which is harmless for the argument) we find that $c_1(\ca{L})\cdot c_1(H)^{\on{dim}\ca{S}_K - 1} >0$ as required. 

\item[Step 3:] comparing metrics on $\ca{L}$. 

Since $\ca{S}_K$ is projective there exists a continuous metric on the line bundle $\ca{L}$ (eg. write $\ca{L}$ as a difference of very ample line bundles, to which we can pull back the Fubini-Study metric). This metric is far from unique, but (by compactness) any two such metrics have bounded difference. Write $\aabs{-}_c$ for one such metric. The bundle $\ca{L}$ also comes with a natural metric $\aabs{-}_\ca{P}$ over $S$ given by pulling back the unique rigidified translation-invariant metric on the Poincar\'e bundle. We do not know\footnote{The metric does extend continuously if $\on{dim}\ca{S}_K = 1$ by \cite{Holmes2013Asymptotics-of-}, it does not extend continuously in general in higher dimension (due to height-jumping), and in the present case (when we have a N\'eron model) we do not know whether it extends. } if this metric extends continuously to the whole of $\ca{S}_K$, but nonetheless we will see in the next paragraph how to use results from \cite{Ignacio-Burgos-Gil2015The-singulariti} to show that $\aabs{-}_\ca{P}$ has bounded difference from $\aabs{-}_c$. 

 By \cite[theorem 1.1 (1)]{Ignacio-Burgos-Gil2015The-singulariti} the logarithm of the norm in $\aabs{-}_\ca{P}$ of a local generating section of $\ca{L}$ differs by a bounded amount from a function of the form
\begin{equation*}
q(\log(\abs{z_1})\ldots, \log(\abs{z_n}))
\end{equation*}
where $z_1, \ldots, z_n$ are suitably chosen local analytic coordinates, and $q$ is some homogeneous rational function of degree 1. However, using that $C_{\ca{S}_K}/\ca{S}_K$ is aligned one can check (using test curves in $\ca{S}_K$) that this rational function $q$ is in fact linear. Thus we see that $\aabs{-}_\ca{P}$ and $\aabs{-}_c$ have bounded difference. 

 Associated to the metrised line bundles $(\ca{L}, \aabs{-}_c)$ and $(\ca{L}, \aabs{-}_\ca{P})$ we get height functions $\on{h}^c_\ca{L}$ and $\on{h}^\ca{P}_\ca{L}\colon S(\bar{K}) \ra \bb{R}$ respectively. Then $\on{h}^c_{\ca{L}}$ is weakly non-degenerate by \ref{lem:very_effective_implies_weakly_non_degen} since $\on{h}^0(\ca{S}_K, \ca{L}^{\otimes n}) \ge 2$. Because $\on{h}^c_\ca{L}$ and $\on{h}^\ca{P}_\ca{L}$ have bounded difference (cf. \ref{rem:bounded_differences}) we deduce that $\on{h}^\ca{P}_\ca{L}$ is also weakly non-degenerate. 

\item[Step 4:] concluding the proof using that the jump is bounded. 

To conclude the proof of the theorem, we compare three functions from $S(\bar{K})$ to $\bb{R}$:
\begin{enumerate}
\item $\hat{\on{h}}_\sigma$ sending $s \in S(\bar{K})$ to the N\'eron-Tate height of $\sigma(s)$; 
\item $j \colon S(\bar{K})\ra\bb{R}$ sending $s \in S(\bar{K})$ to $\frac{\on{deg}j(\sigma, \sigma, \bar{s})}{[\kappa(s):K]}$ where $j(\sigma, \sigma, \bar{s})$ is the jump corresponding to $\sigma$ and $\bar{s}$;
\item the height $\on{h}^\ca{P}_{\ca{L}}$, 
\end{enumerate}
and by construction these satisfy 
\begin{equation*}
\hat{\on{h}}_\sigma = \on{h}^\ca{P}_{\ca{L}} - j. 
\end{equation*}
We have seen above that $\on{h}^\ca{P}_{\ca{L}}$ is weakly non-degenerate, and $\abs{j}$ is bounded by \ref{thm:generic_alignment_bounds_jump}, so we see that $\hat{\on{h}}_\sigma$ is also weakly non-degenerate as required. 
\end{enumerate}
\end{proof}

Note that one of the main results of \cite{David-Holmes2014Neron-models-an} is the nonnegativity of $j$ (as conjectured in \cite{Hain2013Normal-function}), but for our purposes this does not seem very helpful, since we want to deduce positivity of $\hat{\on{h}}_\sigma$ from positivity of $\on{h}_\ca{L}$. In the presence of a N\'eron model on the generic fibre we can show that $j$ is bounded (which is sufficient), but without this assumption it seems hard to control $j$. 

It is interesting to compare our result to \cite[theorem 1.3.5]{Zhang2010Gross-Schoen-cy} where Zhang proves a similar result where the section $\sigma$ of the jacobian is replaced by the Gross-Schoen cycle on the product of $C$ with itself, and where the family of curves $C/S$ is assumed to be smooth. It seems reasonable to speculate that it might be possible to generalise Zhang's result to the case where $C/S$ is assumed only to be regular and aligned, rather than smooth.

\bibliographystyle{alpha} 
\bibliography{../../prebib.bib}

\end{document}